\documentclass[11pt, english]{amsart}
\usepackage{latexsym}
\usepackage{amscd}
\usepackage{amssymb}
\usepackage[all, cmtip]{xy}
\usepackage{amsmath}
\usepackage{graphics, graphicx, calc, tikz, tkz-graph, bm, babel, ragged2e, enumerate, etoolbox, xcolor, caption, url}
\usetikzlibrary{decorations.markings, arrows}
\DeclareMathAlphabet{\eusm}{OT1}{eusm}{m}{n}
\newtheorem{theorem}{Theorem}[section]
\newtheorem{prop}[theorem]{Proposition}

\newtheorem{cor}[theorem]{Corollary}
\newtheorem{lem}[theorem]{Lemma}
\newtheorem{example}[theorem]{Example}
\newtheorem{remark}[theorem]{Remark}

\usetikzlibrary{decorations.markings}
\tikzstyle{vertex}=[circle, draw, fill=black, inner sep=0pt, minimum size=6pt]

}}}]}

\begin{document}
\title[Leavitt path algebras with bounded index of nilpotence]{Leavitt path algebras with bounded index of nilpotence and simple modules over them}
\author{Kulumani M. Rangaswamy}
\address{Departament of Mathematics, University of Colorado at Colorado Springs, Colorado-80918, USA}
\email{krangasw@uccs.edu}
\author{Ashish K. Srivastava}
\address{Department of Mathematics and Statistics, St. Louis University, St.
Louis, MO-63103, USA} \email{asrivas3@slu.edu}
\thanks{The work of the second author is partially supported by a grant from Simons Foundation (grant number 426367).}
\maketitle

\begin{abstract}
In this paper we completely describe graphically Leavitt path algebras with bounded index of nilpotence and show that each graded simple module $S$ over a Leavitt path algebra with bounded index of nilpotence is graded $\Sigma$-injective, that is, $S^{(\alpha)}$ is graded injective for any cardinal $\alpha$. Furthermore, we characterize Leavitt path algebras over which each simple module is $\Sigma$-injective. We have shown that each simple module over a Leavitt path algebra $L_K(E)$ is $\Sigma$-injective if and only if the graph $E$ contains no cycles, and there is a positive integer $d$ such that the length of any path in $E$ is less than or equal to $d$ and the number of distinct paths ending at any vertex $v$ (including $v$) is less than or equal to $d$.   
\end{abstract}

\bigskip

\bigskip

\noindent Leavitt path algebras are algebraic analogues of graph C*-algebras and are also natural generalizations of Leavitt algebras of type $($$1, n$$)$ constructed in \cite{L}. The objective of this paper is to characterize Leavitt path algebras with bounded index of nilpotence and study simple modules over them. The initial organized attempt to study the module theory over Leavitt path algebras was done in \cite{AB} where the simply presented modules over a Leavitt path algebra $L_{K}(E)$ of a finite graph $E$ were described.
As an important step in the study of the modules over $L_{K}(E)$ for an
arbitrary graph $E$, the simple $L_{K}(E)$-modules and\ also Leavitt path
algebras with simple modules of special types, have recently been investigated
in a series of papers (see e.g. \cite{AR1}, \cite{C}, \cite{R1}, \cite{HR}). In this paper we focus on the question when (graded) simple modules over Leavitt path algebras are not only (graded) injective but they are (graded) $\Sigma$-injective and we link this question to the classical ring theoretic properties such as bounded index of nilpotence and direct finiteness.   

A ring $R$ is said to have \textit{bounded index of}
\textit{nilpotence} if there is a positive integer $n$ such that $x^{n}=0$
for all nilpotent elements $x$ in $R$. If $n$ is the least such integer then $R$ is said to have {\it index of nilpotence $n$}. We completely describe graphically the Leavitt path algebras with bounded index of nilpotence. We show that a Leavitt
path algebra $L_K(E)$ of an arbitrary graph $E$ has index of nilpotence $n$ if and only if no cycle in the graph $E$ has an exit, every path in $E$ has at most $n$ distinct vertices and every vertex in $E$ is the range of at most $n$ distinct paths. In this case, $L$ becomes a subdirect product of
matrix rings $M_{t}(K)$ or $M_{t}(K[x,x^{-1}])$ of finite order $t\leq n$.
Interestingly, it turns out that over a Leavitt path algebra with bounded index of nilpotence, each graded simple module is graded $\Sigma$-injective and a Leavitt path algebra with bounded index of nilpotence is directly-finite. We construct examples that show the converses of these statements do not hold. 

Recall that a ring $R$ is called a \textit{right $($left$)$ }$V$\textit{-ring} if
every simple right (left) $R$-module is injective. Kaplansky showed that a
commutative ring $R$ is a $V$-ring if and only $R$ is von Neumann
regular. It is known that this equivalence no longer holds if $R$ is not
commutative (see \cite{F}). A ring $R$ is called a \textit{right $($left$)$ }%
$\Sigma$\textit{-}$V$\textit{ring}, if every simple right (left) $R$-module $S$
is $\Sigma$-injective, that is, any direct sum $S^{(\alpha)}$ of $\alpha$ copies of $S$ is injective. For properties of $V$-rings and $\Sigma$-$V$ rings, see (\cite{F}, \cite{Lam}, \cite{AS}). In this paper, we
characterize, both graphically and algebraically, the Leavitt path algebra $L$
of an arbitrary graph $E$ which is a $\Sigma$-$V$ ring. Specifically, we show that
$L$ is a $\Sigma$-$V$ ring if and only if the graph $E$ contains no cycles, and there is a positive integer $d$ such that the length of any path in $E$ is less than or equal to $d$ and the number of distinct paths ending at any vertex $v$ (including $v$) is less than or equal to $d$. Algebraically, this is equivalent to $L$ being von
Neumann regular ring and a subdirect product of an arbitrary number of matrix
rings over $K$ of finite order $n$ bounded above by a positive integer $d$. If $E$ is a
row-finite graph, then $L$ actually decomposes as a direct sum of matrix rings,
$L=%
{\displaystyle\bigoplus\limits_{i\in I}}
M_{n_{i}}(K)$, where $I$ is an arbitrary index set, and $d$ is a fixed positive
integer such that $n_{i}\leq d$ for all $i\in I$.

\bigskip

\section{Preliminaries}

\noindent For the general notation, terminology and results in Leavitt path algebras, we
refer to \cite{AArS}, \cite{R} and \cite{T}. We will be using some of the
needed results in associative rings, von Neumann regular rings and modules
from \cite{Goodearl} and \cite{Lam}. We give below an outline of some of the needed basic
concepts and results.

A (directed) graph $E=(E^{0},E^{1},r,s)$ consists of two sets $E^{0}$ and
$E^{1}$ together with maps $r,s:E^{1}\rightarrow E^{0}$. The elements of
$E^{0}$ are called \textit{vertices} and the elements of $E^{1}$
\textit{edges}.

A vertex $v$ is called a \textit{sink} if it emits no edges and a vertex $v$
is called a \textit{regular} \textit{vertex} if it emits a non-empty finite
set of edges. An \textit{infinite emitter} is a vertex which emits infinitely
many edges. For each $e\in E^{1}$, we call $e^{\ast}$ a ghost edge. We let
$r(e^{\ast})$ denote $s(e)$, and we let $s(e^{\ast})$ denote $r(e)$.
A\textit{\ path} $\mu$ of length $n>0$ is a finite sequence of edges
$\mu=e_{1}e_{2}\cdot\cdot\cdot e_{n}$ with $r(e_{i})=s(e_{i+1})$ for all
$i=1,\cdot\cdot\cdot,n-1$. In this case $\mu^{\ast}=e_{n}^{\ast}\cdot
\cdot\cdot e_{2}^{\ast}e_{1}^{\ast}$ is the corresponding ghost path. A vertex
is considered a path of length $0$.

A path $\mu$ $=e_{1}\dots e_{n}$ in $E$ is \textit{closed} if $r(e_{n}%
)=s(e_{1})$, in which case $\mu$ is said to be \textit{based at the vertex
}$s(e_{1})$. A closed path $\mu$ as above is called \textit{simple} provided
it does not pass through its base more than once, i.e., $s(e_{i})\neq
s(e_{1})$ for all $i=2,...,n$. The closed path $\mu$ is called a
\textit{cycle} if it does not pass through any of its vertices twice, that is,
if $s(e_{i})\neq s(e_{j})$ for every $i\neq j$.

A graph $E$ is said to satisfy \textit{Condition (K)}, if any vertex $v$ on a
closed path $c$ is also the base of a another closed path $c^{\prime}%
$different from $c$.

An \textit{exit }for a path $\mu=e_{1}\dots e_{n}$ is an edge $e$ such that
$s(e)=s(e_{i})$ for some $i$ and $e\neq e_{i}$.

If there is a path from vertex $u$ to a vertex $v$, we write $u\geq v$. A
subset $D$ of vertices is said to be \textit{downward directed }\ if for any
$u,v\in D$, there exists a $w\in D$ such that $u\geq w$ and $v\geq w$. A
subset $H$ of $E^{0}$ is called \textit{hereditary} if, whenever $v\in H$ and
$w\in E^{0}$ satisfy $v\geq w$, then $w\in H$. A hereditary set is
\textit{saturated} if, for any regular vertex $v$, $r(s^{-1}(v))\subseteq H$
implies $v\in H$.

Given an arbitrary graph $E$ and a field $K$, the \textit{Leavitt path algebra
}$L_{K}(E)$ is defined to be the $K$-algebra generated by a set $\{v:v\in
E^{0}\}$ of pair-wise orthogonal idempotents together with a set of variables
$\{e,e^{\ast}:e\in E^{1}\}$ which satisfy the following conditions:

(1) \ $s(e)e=e=er(e)$ for all $e\in E^{1}$.

(2) $r(e)e^{\ast}=e^{\ast}=e^{\ast}s(e)$\ for all $e\in E^{1}$.

(3) (The ``CK-1 relations") For all $e,f\in E^{1}$, $e^{\ast}e=r(e)$ and
$e^{\ast}f=0$ if $e\neq f$.

(4) (The ``CK-2 relations") For every regular vertex $v\in E^{0}$,
\[
v=\sum_{e\in E^{1},s(e)=v}ee^{\ast}.
\]
Every Leavitt path algebra $L_{K}(E)$ is a $%
\mathbb{Z}
$\textit{-graded algebra}, namely, $L_{K}(E)=%
{\displaystyle\bigoplus\limits_{n\in\mathbb{Z}}}
L_{n}$ induced by defining, for all $v\in E^{0}$ and $e\in E^{1}$, $\deg
(v)=0$, $\deg(e)=1$, $\deg(e^{\ast})=-1$. Here the $L_{n}$ are abelian
subgroups satisfying $L_{m}L_{n}\subseteq L_{m+n}$ for all $m,n\in%
\mathbb{Z}
$. Further, for each $n\in%
\mathbb{Z}
$, the \textit{homogeneous component }$L_{n}$ is given by
\[
L_{n}=\{%
{\textstyle\sum}
k_{i}\alpha_{i}\beta_{i}^{\ast}\in L:\text{ }|\alpha_{i}|-|\beta_{i}|=n\}.
\]
Elements of $L_{n}$ are called \textit{homogeneous elements}. An ideal $I$ of
$L_{K}(E)$ is said to be a \textit{graded ideal} if $I=$ $%
{\displaystyle\bigoplus\limits_{n\in\mathbb{Z}}}
(I\cap L_{n})$. If $A,B$ are graded modules over a graded ring $R$, we write
$A\cong_{gr}B$ if $A$ and $B$ are graded isomorphic and we write $A\oplus
_{gr}B$ to denote a graded direct sum. We will also be using the usual grading
of a matrix of finite order. For this and for the various properties of graded
rings and graded modules, we refer to \cite{H-1}, \cite{HR} and \cite{NV}.

A \textit{breaking vertex }of a hereditary saturated subset $H$ is an infinite
emitter $w\in E^{0}\backslash H$ with the property that $0<|s^{-1}(w)\cap
r^{-1}(E^{0}\backslash H)|<\infty$. The set of all breaking vertices of $H$ is
denoted by $B_{H}$. For any $v\in B_{H}$, $v^{H}$ denotes the element
$v-\sum_{s(e)=v,r(e)\notin H}ee^{\ast}$. Given a hereditary saturated subset
$H$ and a subset $S\subseteq B_{H}$, $(H,S)$ is called an \textit{admissible
pair.} Given an admissible pair $(H,S)$, the ideal generated by $H\cup
\{v^{H}:v\in S\}$ is denoted by $I(H,S)$. It was shown in \cite{T} that the
graded ideals of $L_{K}(E)$ are precisely the ideals of the form $I(H,S)$ for
some admissible pair $(H,S)$. Moreover, $L_{K}(E)/I(H,S)\cong L_{K}%
(E\backslash(H,S))$. Here $E\backslash(H,S)$ is a \textit{Quotient graph of
}$E$ where $(E\backslash(H,S))^{0}=(E^{0}\backslash H)\cup\{v^{\prime}:v\in
B_{H}\backslash S\}$ and $(E\backslash(H,S))^{1}=\{e\in E^{1}:r(e)\notin
H\}\cup\{e^{\prime}:e\in E^{1}$ with $r(e)\in B_{H}\backslash S\}$ and $r,s$
are extended to $(E\backslash(H,S))^{0}$ by setting $s(e^{\prime})=s(e)$ and
$r(e^{\prime})=r(e)^{\prime}$.

We will also be using the fact that the Jacobson radical (and in particular,
the prime/Baer radical) of $L_{K}(E)$ is always zero (see \cite{AArS}). It is known (see \cite{R}) that if $P$ is a prime ideal of $L$ with $P\cap
E^{0}=H$, then $E^{0}\backslash H$ is downward directed.

A \textit{maximal tail }is a subset $M$ of $E^0$ satisfying the following three properties:
\begin{enumerate}

\item $M$ is downward directed, that is, for any pair of vertices $u,v\in M$, there exists a
$w\in M$ such that $u\geq w,v\geq w$;

\item If $u\in M$ and $v\in E^{0}$ satisfies $v\geq u$, then $v\in M$;

\item If $u\in M$ emits edges, there is at least one edge $e$ with $s(e)=u$
and $r(e)\in M$.
\end{enumerate}

Let $\Lambda$ be an infinite set and $R$, a ring. Then $M_{\Lambda}(R)$ denotes the ring of $\Lambda\times\Lambda$ matrices in which all except at most finitely many entries are non-zero.

\bigskip

\section{Leavitt path algebras having bounded index of nilpotence}

\noindent Recall that a ring $R$ is said to have \textit{bounded index of}
\textit{nilpotence} if there is a positive integer $n$ such that $x^{n}=0$
for all nilpotent elements $x$ in $R$. If $n$ is the least such integer then $R$ is said to have {\it index of nilpotence $n$}. In this section, we characterize Leavitt path algebras having bounded index of nilpotence and show their connection to the behavior of simple modules over them.

\begin{remark} \rm
\label{Thm7-2} We shall be using a version of Theorem 7.2 in \cite{Goodearl}
showing that condition (a) of that theorem implies conditions (c) and (d).
This theorem assumes that the ring is a von Neumann regular ring with
identity. But an examination of the proof shows that the same proof works for
rings without identity and that the proof of (a)$\implies$(c) does not use the von Neumann regularity of $R$ and that a version of (d)
as stated below does not need von Neumann regularity. So we re-state, as a
separate proposition, a part of Theorem 7.2 of \cite{Goodearl} which is valid
for any ring.
\end{remark}

\begin{prop}
\label{Goodearl Thm} $($Theorem 7.2, \cite{Goodearl}$)$ Suppose $R$ is a ring not
necessarily having a multiplicative identity. If $R$ has index of
nilpotence $n$, then the following properties hold:
\begin{enumerate}
\item If $e_{1},\cdot\cdot\cdot,e_{n},e_{n+1}$ are orthogonal idempotents in
$R$, then $e_{1}Re_{2}R\cdot\cdot\cdot e_{n}Re_{n+1}=0$.

\item $R$ contains no direct sums of $n+1$ non-zero pair-wise isomorphic right
ideals $A_{1},\cdot\cdot\cdot,A_{n},A_{n+1}$, where, for each $i$,
$A_{i}=\epsilon_{i}R$ and the $\epsilon_{i}$ are orthogonal idempotents.
\end{enumerate}
\end{prop}

\noindent A well-known useful fact (\cite{Goodearl}) is that when $K$ is a field or an
integral domain, the ring $M_{n}(K)$ of $n\times n$ matrices over $K$ has index of nilpotence $n$.

\begin{prop}
\label{bddIdx = > NE} Suppose $E$ is an arbitrary graph and $L:=L_{K}(E)$ has index of nilpotence $n$. Then no cycle in $E$ has an exit.
\end{prop}

\begin{proof}
Suppose there is a cycle $c$ having an exit $f$ at a vertex $v$. Clearly, for
each $n\geq0$, $\epsilon_{n}=c^{n}(c^{\ast})^{n}$ is an idempotent and for
each $n$, $\epsilon_{n}\epsilon_{n+1}=\epsilon_{n+1}=\epsilon_{n+1}%
\epsilon_{n}$. So $\epsilon_{n}L\supseteq\epsilon_{n+1}L$. This inclusion is
strict, because, if $\epsilon_{n}=\epsilon_{n+1}a$ for some $a\in L$, then
$f^{\ast}(c^{\ast})^{n}\epsilon_{n}=f^{\ast}(c^{\ast})^{n}\epsilon_{n+1}a$.
This equation becomes $f^{\ast}(c^{\ast})^{n}c^{n}(c^{\ast})^{n}=f^{\ast
}(c^{\ast})^{n}c^{n+1}(c^{\ast})^{n+1}a$ and, by CK-1 condition, it simplifies
to $f^{\ast}(c^{\ast})^{n}=f^{\ast}c(c^{\ast})^{n+1}=0$, a contradiction.
Thus, for every $n$, we have a non-trivial direct decomposition $\epsilon
_{n}L=\epsilon_{n+1}L\oplus(\epsilon_{n}-\epsilon_{n+1})L$. Moreover, for
every $n\geq0$, it is easy to see trhat $\epsilon_{n}L\cong\epsilon_{n+1}L$
under the map $\theta:\epsilon_{n}a\longmapsto\epsilon_{n+1}\epsilon_{n}a$.
So, for each $n\geq0$, $\epsilon_{n}L\cong\epsilon_{n+1}L$ and $\epsilon
_{n}L=\epsilon_{n+1}L\oplus(\epsilon_{n}-\epsilon_{n+1})L$. \ Thus we get a
series of decompositions $vL=B_{1}\oplus A_{1}$, where $B_{1}=$ $\epsilon
_{1}L\cong vL$, $B_{1}=B_{2}\oplus A_{2}$, where $B_{2}\cong B_{1}$,
$A_{2}\cong A_{1}$, $\cdot\cdot\cdot,B_{n}=B_{n+1}\oplus A_{n+1}$, with
$B_{n+1}\cong B_{n}$ and $A_{n+1}\cong A_{n}$, etc. Now the $B_{i}\cong vL$
for all $i$ and $A_{1}\cong A_{2}\cong\cdot\cdot\cdot\cong A_{n}\cong
A_{n+1}\cong\cdot\cdot\cdot$. Note that, if $\epsilon_{1}=u_{2}+v_{2}$, with
$u_{2}\in B_{2}$ and $v_{2}\in A_{2}$, then $u_{2},v_{2}$ are orthogonal
idempotents, $\epsilon_{1}u_{2}=u_{2}=u_{2}\epsilon_{1}$, $\epsilon_{1}%
v_{2}=v_{2}=v_{2}\epsilon_{1}$, $B_{2}=u_{2}L$ and $A_{2}=v_{2}L$. Proceeding
like this we get, for all $n\geq1$, $B_{n+1}=u_{n+1}L$ and $A_{n+1}=v_{n+1}L$
where $u_{n+1}$, $v_{n+1}$ are orthogonal idempotents, $u_{n}=u_{n+1}+v_{n+1}%
$. Also $A_{1}=(v-\epsilon_{1})L$. Let $v_{1}=v-\epsilon_{1}$. It is then
clear that $v_{1},\cdot\cdot\cdot,v_{n},v_{n+1}$ are orthogonal idempotents
and $A_{i}=v_{i}L$, for all $i=1,\cdot\cdot\cdot,n+1$. Thus the right ideals
$A_{i}=\epsilon_{i}L$ are independent and so $vL$ contains the direct sum
$A_{1}\oplus A_{2}\oplus\cdot\cdot\cdot\oplus A_{n}\oplus A_{n+1}$ of $n+1$
isomorphic right ideals. By Propopsition \ref{Goodearl Thm}, this is a
contradiction. Hence no cycle in the graph $E$ has an exit.
\end{proof}

Now we are ready to characterize Leavitt path algebras with bounded index of nilpotence. 

\begin{theorem}
\label{bdd index} Let $E$ be an arbitrary graph. Then the following properties
are equivalent for $L:=L_{K}(E)$:

(a) $L$ has index of nilpotence at most $n$;

(b) No cycle in the graph $E$ has an exit, there is a fixed positive integer
$n$ such that the number of distinct vertices in any path in $E$ is less than or equal to $n$
and that the number of distinct paths that end at any given vertex $v$
(including $v)$ is less than or equal to $n$;

(c) There is a fixed positive integer $n$ such that for every graded prime
ideal $Q$ of $L$, $L/Q$ is graded isomorphic to $M_{t}(K)$ or $M_{t}%
(K[x,x^{-1}])$ for integer $t\leq n$ under the usual grading of matrices.
\end{theorem}

\begin{proof}
Assume (a). Suppose, by way of contradiction, there is a path $p$ that
contains $n+1$ distinct vertices $v_{1},\cdot\cdot\cdot,v_{n},v_{n+1}$
occurring in this order in the path $p$ with $v_{1}=s(p)$. So we can write $p$
as a concatenation of paths $p=p_{1}p_{2}\cdot\cdot\cdot p_{n+1}$ such that
$v_{i}=s(p_{i})$ for $i=1,\cdot\cdot\cdot,n+1$. Now $v_{1}L,\cdot\cdot
\cdot,v_{n+1}L$ are $n+1$ independent right ideals of $L$ and so, by Theorem
7.2 in \cite{Goodearl}, their product $v_{1}L\cdot v_{2}L\cdot\cdot\cdot
v_{n+1}L=0$. This is a contradiction, since
\[
0\neq p=p_{1}p_{2}\cdot\cdot\cdot p_{n+1}=v_{1}p_{1}\cdot v_{2}p_{2}\cdot
\cdot\cdot v_{n+1}p_{n+1}\in v_{1}L\cdot v_{2}L\cdot\cdot\cdot v_{n+1}L.
\]

Let $Q=I(H,S)$ be a graded prime ideal of $L$, where $H=Q\cap E^{0}$ and
$S\subseteq B_{H}$.

Suppose $E\backslash(H,S)$ contains a sink $w$. By Proposition
\ref{bddIdx = > NE}, no cycle in $E$ and hence in $E\backslash(H,S)$ has an
exit. As $E\backslash(H,S)^{0}$ is downward directed, $w$ is the only sink in
$E\backslash(H,S)$ and $u\geq w$ for every $u\in E\backslash(H,S)^{0}$. This
implies that $E\backslash(H,S)$ must be acyclic. Thus every path in
$E\backslash(H,S)$ has no repeated vertices and as there are at most $n$
distinct vertices in any path in $E\backslash(H,S)$, we conclude that the
length of any path in $E\backslash(H,S)$ is at most $n$. We then appeal to
Proposition 3.5 in \cite{AAS-1} to conclude that $L/Q\cong L_{K}%
(E\backslash(H,S))\cong M_{\Lambda}(K)$, where $\Lambda$ denotes the set of
distinct paths that end at $w$ (including $w$). If $|\Lambda|>n$, then
$M_{\Lambda}(K)$ will contain, as a subring, the matrix ring $M_{t}(K)$ for
some integer $t>n$. As $M_{t}(K)$ has index of nilpotency $t>n$, this
contradicts the fact that $L/Q$ has bounded index of nilpotence $n$. Thus
$|\Lambda|=r\leq n$. This means the number of distinct paths that end at $w$
(including $w$) is $r\leq n$. Now for any vertex $v$ in $E\backslash(H,S)$,
since there is a path from $v$ to $w$, every path that ends at $v$ extends to
a path that end at $w$. Hence the number of distinct paths in $E\backslash
(H,S)$ that end at $v$ (including $v$) is also less than or equal to $n$. Since these are
precisely the paths in $E$ that end at $v$, we conclude that the number of
distinct paths in $E$ that end at $v$ is at most $n$.

Suppose, $E\backslash(H,S)$ does not contain any sink. Since every path
contains at most $n$ distinct vertices and since no cycle has an exit, every
path in $E\backslash(H,S)$ eventually ends at the base $u$ of a cycle $c$
without exits. By downward directness, $C$ is the only cycle in $E\backslash
(H,S)$ and $u^{\prime}\geq u$ for every $u^{\prime}\in E\backslash(H,S)^{0}$.
We then appeal to Proposition 3.4 of \cite{AAS} to conclude that $L/Q\cong
L_{K}(E\backslash(H,S))\cong M_{\Lambda^{\prime}}(K[x,x^{-1}])$ where
$\Lambda^{\prime}$ is the set of all distinct paths that end at $v$ (including
$v$), but do not contain all the edges in the cycle $c$. As argued in the
preceding paragraph, if $|\Lambda^{\prime}|=r>n$ it will contradict the fact
that $n$ is the bound for nilpotent index for $L$. This means the number $r$
of distinct paths in $E\backslash(H,S)$ that end at $u$ (including $u$) is
less than or equal to $n$. Thus for any vertex $v\in E\backslash(H,S)$, since there is a path
from $v$ to $u$, the number of distinct paths in $E\backslash(H,S)$ (and
therefore in $E$) that end at $v$ (including $v$) is less than or equal to $n$. This
proves (b).

Assume (b). Let $Q=I(H,S)$ be a graded prime ideal of $L$ where $H=Q\cap
E^{0}$ and $S\subseteq B_{H}$. Since $E\backslash(H,S)$ is downward directed,
it contains at most one sink and also at most one cycle, as no cycle has an
exit. Since the number of distinct vertices in any path is less than or equal to $n$, every
path in $E\backslash(H,S)$ must end either at a sink or at a cycle. If
$E\backslash(H,S)$ contains a sink $w$, then $u\geq w$ for every vertex $u\in
E\backslash(H,S)$ and since the number of distinct paths that end at $w$
(including $w$) is less than or equal to $n$, $L/Q\cong L_{K}(E\backslash(H,S))\cong$
$M_{t}(K)$ for some $t\leq n$, by Proposition 3.4 of \cite{AAS}. In (Theorem
4.14,\cite{H-1}) it was shown that this isomorphism induces a graded
isomorphism with $M_{t}(K)$ given the usual grading of matrix rings
(\cite{H-1}). If $E\backslash(H,S)$ does not contain a sink, then necessarily
every path in $E\backslash(H,S)$ ends at a unique cycle $c$ without exits
based at a vertex $v$. Since the number of distinct paths that end at $v$ is
less than or equal to $n$, we appeal to Proposition 3.5 in \cite{AAS-1} to conclude that
$L/Q\cong L_{K}(E\backslash(H,S))\cong$ $M_{t}(K[x,x^{-1}])$ for some $t\leq
n$. Again, by (Theorem 4.20, \cite{H-1}), the last isomorphism induces a
graded isomorphism with $M_{t}(K[x,x^{-1}])$ given the usual grading of
matrices (\cite{H-1}). This proves (c).

Assume (c). Now the intersection $%
{\displaystyle\bigcap}
\{Q:Q$ a graded prime ideal of $L\}=0$, since $%
{\displaystyle\bigcap}
\{P:P$ a prime ideal of $L\}=0$ and every prime ideal $P$ contains a graded
prime ideal $gr(P)$ (see \cite{R}). Thus $L$ is a subdirect product of
$\{L/Q:Q$ a graded prime ideal of $L\}$. For each graded prime ideal $Q$, we
have, by hypothesis, $L/Q\cong$ $M_{t}(K)$ or $M_{t}(K[x,x^{-1}])$ with $t\leq
n$ and so $L/Q$ has index of nilpotence less than or equal to $n$. Then $L$, being a
subdirect product of the rings $L/Q$ for various $Q$, also has index
of nilpotence less than or equal to $n$. This proves (a).
\end{proof}

\begin{remark} \rm
\label{grPrime => grMax} Now, as graded rings, both $M_{t}(K)$ and
$M_{t}(K[x,x^{-1}])$ with $t\leq n$ are graded direct sums of finitely many
isomorphic graded simple modules and so both are graded artinian and simple.
In view of this, Condition (c) of Theorem \ref{bdd index} allows an interesting characterization which we state below.
\end{remark}

\begin{theorem}
\label{Condn c Reform} A Leavitt path algebra $L$ has bounded index of
nilpotence if and only if, for every graded prime ideal $Q$ of $L$, $L/Q$ is a
graded simple (graded) artinian ring.
\end{theorem}

In the case when the graph $E$ is row-finite, we obtain a stronger version of
Theorem \ref{bdd index}, thus leading to a structure theorem.

\begin{theorem}
\label{bddIdx-row-finte} Let $E$ be row-finite graph. Then the following
properties are equivalent for $:=L_{K}(E)$;

(a) $L$ has index of nilpotence at most $n$;

(b) There is a fixed positive integer $n$ and an isomorphism%

\[
L\cong%
{\displaystyle\bigoplus\limits_{i\in I}}
M_{n_{i}}(K)%
{\displaystyle\bigoplus}
{\displaystyle\bigoplus\limits_{j\in J}}
M_{n_{j}}(K[x,x^{-1}])
\]
where $I$ is the set of all sinks in $E$ and $J$ is the set of cycles without
exits in $E$ and, for all $i\in I,j\in J$, $n_{i}\leq n$ and $n_{j}\leq n$.
Thus, in particular, $L$ is graded semi-simple (that is, a direct sum of
graded simple left/right ideals).
\end{theorem}

\begin{proof}
Assume (a). By Theorem \ref{bdd index}, no cycle in $E$ has an exit and every
path in $E$ contains at most $n$ distinct vertices. This means that every path
eventually ends at a sink or at a vertex on a cycle without exits. Since the
number of distinct paths that end at any vertex in $E$ is less than or equal to $n$, we appeal
to Theorem 3.7 and Propositions 3.5 and 3.6 of \cite{AAS-1} to conclude that
\[
L\cong%
{\displaystyle\bigoplus\limits_{i\in I}}
M_{n_{i}}(K)%
{\displaystyle\bigoplus}
{\displaystyle\bigoplus\limits_{j\in J}}
M_{n_{j}}(K[x,x^{-1}])
\]
where $n_{i}$ and $n_{j}$ are less than or equal to $n$. This proves (b).

Now (b)$\implies$(a) follows from the fact that when $n_{i}$ and $n_{j}$ are less than or equal to $n$, both the
matrix rings $M_{n_{i}}(K)$ and $M_{n_{j}}(K[x,x^{-1}])$ have index of
nilpotence less than or equal to $n$.
\end{proof}

As a consequence of the Theorem \ref{bdd idx}, we have the following for graded simple modules over a Leavitt path algebra with bounded index of nilpotence. 
 
\begin{prop}
\label{bdd idx = > graded sigmaV} If a Leavitt path algebra $L$ has bounded
index of nilpotence, then $L$ is a graded $\Sigma$-$V$ ring, that is, each graded simple module over $L$ is graded $\Sigma$-injective. But the converse does
not hold.
\end{prop}

\begin{proof}
Suppose $S$ be a graded simple right $L$-module. Then $P=ann_{L}(S)$ is a
graded prime ideal of $L$. By Theorem \ref{bdd index}, there is a graded
isomorphism $L/P\cong M_{t}(K)$ or $M_{t}(K[x,x^{-1}])$ for some $t\leq n$
with the matrices $M_{t}(K)$ and $M_{t}(K[x,x^{-1}])$ being given the matrix
gradings as indicated in \cite{H-1}. As $L/P$ is graded semi-simple (that is
direct sum of graded simple modules), $S$ is $\Sigma$-injective as a graded
right $L/P$-module. As $L$ is graded von Neumann regular, L/P is flat as an
$L$-module. Then, using the graded version of the well-known injective
producing lemma (see Chapter I, Corollary 3.6A in \cite{Lam}), $S$ is graded
$\Sigma$-injective as a right $L$-module. This proves that $L$ is a graded
$\Sigma$-$V$ ring.

Now for an infinite set $\Lambda$, $M_{\Lambda}(K[x,x^{-1}])$ is a graded
semi-simple ring (that is, a graded direct sum of graded simple modules)
(\cite{HR}) and so is a graded $\Sigma$-$V$ ring. But it does not have bounded
index of nilpotence as it includes subrings isomorphic to $M_{n}(K[x,x^{-1}])$
for various positive integers $n$.
\end{proof}

We give below an example of a Leavitt path algebra $L$ which is graded $\Sigma$-$V$ ring, but $L$ does not have bounded index of nilpotence.

\begin{example} \rm
\label{bdd idx no sigmaV} Consider the following graph $F$ consisting of two
cycles $g$ and $c$ connected by an edge:%

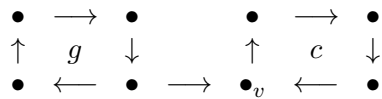
\[%
\begin{array}
[c]{ccccccc}%
\bullet & \longrightarrow & \bullet &  & \bullet & \longrightarrow & \bullet\\
\uparrow & g & \downarrow &  & \uparrow & c & \downarrow\\
\bullet & \longleftarrow & \bullet & \longrightarrow & \bullet_{v} &
\longleftarrow & \bullet
\end{array}
\]

\noindent Now $F$ is downward directed, $c$ is a cycle without exits and the various
powers of the cycle $g$ give rise to infinitely many distinct paths that end
at the base $v$ of the cycle $c$. Hence $L_{K}(F)\cong M_{\infty}%
(K[x,x^{-1}])$ (by Proposition 3.4 of \cite{AAS}) and is graded semi-simple.
Hence each graded simple module over $L_{K}(F)$ is graded $\Sigma$-injective. But $L_{K}(F)$ does not have
bounded index of nilpotence, as $M_{\infty}(K[x,x^{-1}])$ contains subrings
isomorphic to $M_{n}(K[x,x^{-1}])$ for every positive integers $n$.
\end{example}

\begin{remark} \rm
It was shown in \cite{V} that if no cycle in a graph $E$ has an exit, then the
Leavitt path algebra $L:=L_{K}(E)$ is directly-finite. In view of Theorem
\ref{bdd index}, we thus conclude that if $L$ has bounded index of nilpotence,
then $L$ is directly-finite. 
\end{remark}

The next example shows that if a Leavitt path algebra is directly-finite, it need
not have bounded index of nilpotence.

\begin{example} \rm
\label{DF need not have bddIdx} Let $E$ be a graph consisting of infinitely
many edges $\{g,e_{1},\cdot\cdot\cdot,e_{n},\cdot\cdot\cdot\}$ such that
$v=s(g)=r(g)=r(e_{i})$ and $s(e_{i})=v_{i}$ for all $i\geq1$. Since the only
cycle (loop) $g$ has no exit, $L_{K}(E)$ is directly-finite. But $L_{K}(E)$
does not have bounded index of nilpotence since $L_{K}(E)\cong M_{\infty
}(K[x,x^{-1}])$.
\end{example}

\section{Characterization of $\Sigma$-$V$ Leavitt path algebras}

\bigskip

\noindent In this section we provide both graphical and algebraic characterizations for each simple module over a Leavitt path algebra $L$ to be $\Sigma$-injective. In this case, $L$ becomes
von Neumann regular directly finite of bounded nilpotent index and is a
subdirect product of matrix rings of finite order. 

Recall that a ring $R$ is called a right $V$\textit{-ring} if each simple right $R$-module
is injective \cite{Vil}. It is a well-known that a commutative ring is von
Neumann regular if and only if it is a $V$-ring. However, in the case of
noncommutative setting, the classes of von Neumann regular rings and $V$-rings
are quite independent. A ring $R$ is called a right $\Sigma$\textit{-}%
$V$\textit{ ring} if each simple right $R$-module is $\Sigma$-injective (see
\cite{GV}, \cite{Baccella} and \cite{AS}). Recall that a module $M$ is called
$\Sigma$\textit{-injective }if $M^{(\alpha)}$ is injective for any cardinal
$\alpha$. It may be noted here that the direct sums of injective modules provide a great deal of information about the behavior of ring. For example, it is known that a ring $R$ is right noetherian if and only if each injective right $R$-module is $\Sigma$-injective. Clearly, a $\Sigma$-$V$ ring is a $V$-ring, however, there are
examples of $V$-rings that are not $\Sigma$-$V$ rings.

In order to characterize $\Sigma$-$V$ Leavitt path algebras, we begin with a series of preparatory propositions and lemmas. 

The proposition below is proved for rings with identity in \cite{AS}, but we give here a much shorter argument and we prove it more generally for rings with local units. Recall that a ring $R$ with identity is called directly-finite if for each $a, b\in R$, $ab=1$ implies $ba=1$. Motivated by this, a ring $R$ with local units is said to be \textit{directly-finite} if for
every $a,b\in R$ and an idempotent $u\in R$ such that $ua=au=a$ and $ub=bu=b$,
we have that $ab=u$ implies $ba=u$. Equivalently, for every local unit $u$,
$uR$ is not isomorphic to any proper direct summand. This is same thing as
saying that the corner ring $uRu$ is a directly-finite ring with identity. As was
explained in \cite{V}, if $R$ is a directly-finite ring with identity $1$, it
is also directly finite as defined above for rings with local units.
 
\begin{prop} \label{df}
Let $R$ be a ring with local units and assume that $R$ is a right $\Sigma$-$V$ ring. Then $R$ is directly-finite.
\end{prop}

\begin{proof} Assume to the contrary that $R$ is not directly-finite. Then there exists a local unit $u$ such that $uR=S_{1}\oplus T_{1}$
with $uR\cong S_{1}$. Then $S_{1}=S_{2}\oplus T_{2}$ with $S_{2}\cong
S_{1},T_{2}\cong T_{1}$. Procceding like this, we get an infinite number of
idependent (cyclic) direct summands $T_{1}\cong T_{2}\cong\cdot\cdot\cdot\cong
T_{n}\cong\cdot\cdot\cdot$ . For each $n$, choose a maximal submodule $M_{n}$
of $T_{n}$ so that $T_{i}/M_{i}\cong T_{j}/M_{j}$ for all $i,j$. Then $(%
{\displaystyle\bigoplus\limits_{n\geq1}}
T_{n})/(%
{\displaystyle\bigoplus\limits_{n\geq1}}
M_{n})\cong%
{\displaystyle\bigoplus\limits_{n\geq1}}
(T_{n}/M_{n})$ is a direct sum of isomorphic simple right $R$-modules and so
is injective. Consequently it is a direct summand of $uR/(%
{\displaystyle\bigoplus\limits_{n\geq1}}
M_{n})$ and hence is cyclic, a contradiction. This shows that $R$ must be directly-finite.
\end{proof}

As a consequence, we have the following graphical conclusion for a $\Sigma$-$V$ Leavitt path algebra. 

\begin{cor} \label{df=>NE}
If $L:=L_{K}(E)$ is a $\Sigma$-$V$ ring, then no closed path in $E$ has an exit.
\end{cor}

\begin{proof}
Let $L:=L_{K}(E)$ be a $\Sigma$-$V$ ring. We wish to show that no closed path in $E$ has an exit. Suppose, on the contrary, $E$ contains a cycle $c=e_{1}%
\cdot\cdot\cdot e_{n}$ with $s(e_{1})=v=r(e_{n})$ and with an exit $f$ at $v$
(so $s(f)=v$). By the CK-1 relation in the definition of a Leavitt path
algebra, $c^{\ast}c=v$. By the above proposition $L$ is directly-finite and so we have $cc^{\ast}=v$.
Multiplying both sides of the equation by $f^{\ast}$ and using the CK-1
relation, we obtain $0=f^{\ast}cc^{\ast}=f^{\ast}v=f^{\ast}$, a contradiction.
So no cycle in $E$ has an exit.
\end{proof}

Recall that a ring $R$ is called right weakly regular if for each right ideal $I$ of $R$, $I^2=I$. Clearly, every von Neumann regular ring is both right and left weakly regular.   

\begin{lem} \label{wr}
Suppose $R$ is a ring with local units. If $R$ is a right $V$-ring, then $R$ is right weakly regular.
\end{lem}

\begin{proof} Villamayor proved that in a right $V$-ring with identity, every right ideal is an intersection of maximal
right ideals. It may be checked that this result holds for a $V$-ring
$R$ with local units too. Now, let $I$ be a
non-zero right ideal of $R$. If $I\neq I^{2}$, let $a\in I\backslash I^{2}$.
Since $I^{2}$ is an intersection of maximal right ideals, there is a maximal
right ideal $M$ containing $I^{2}$ such that $a\notin M$. Then $R=aR+M$. Let
$u$ be a local unit satisfying $au=ua=a$. Write $u=ax+m$ where $x\in R$ and
$m\in M$. Then $a=ua=axa+ma\in I^{2}+M=M$, a contradiction. Hence $I^{2}=I$
for every right ideal $I$ of $R$. This shows that $R$ is right weakly regular.
\end{proof}

\begin{lem} \label{k} $($Theorem 3.1, \cite{ARS}$)$ Let $L:=L_{K}(E)$ be a Leavitt path algebra of an arbitrary graph $E$. If $L$ is right weakly regular, then the graph $E$ satisfies the Condition (K).
\end{lem}

As a consequence, we have

\begin{prop} \label{vnr}
If $L:=L_{K}(E)$ is a $\Sigma$-$V$ ring, then $E$ is acyclic and consequently, $L$ is von Neumann regular.
\end{prop}

\begin{proof}
Let $L$ be a right $\Sigma$-$V$ ring. We first show that the graph $E$ is acylic. By Corollary \ref{df}, no closed
path in the graph $E$ has an exit. On the other hand, since $L$ is also a
right $V$-ring, Lemma \ref{wr} implies that $L$ is right weakly regular. Now,
in view of the Lemma \ref{k}, we know that the graph $E$ satisfies the Condition
(K) which in particular implies that every cycle in $E$ has an exit. These
contradicting statements imply that the graph $E$ contains no cycles and hence $L$ is von Neumann regular by \cite[Theorem 1]{AR}.
\end{proof}

\begin{remark}
Note that, in general, a $\Sigma$-$V$ ring need not be von Neumann regular. The famous example of Cozzens \cite{Cozzens} is a $\Sigma$-$V$ ring (being a noetherian $V$-domain), but it is not von Neumann regular as it is not artinian.
\end{remark}

\noindent Before proving our next proposition, we note that the Leavitt path algebra $L_K(E)$ of any graph $E$ is non-singular (see \cite{AArS}) and so the maximal right quotient ring $Q^r_{max}(L_K(E))$ of $L_K(E)$ is a von Neumann regular right self-injective ring. It is not difficult to see that any semiprime right self-injective ring possesses an identity element (see for example, page 1085, \cite{FU}). So, although $L_K(E)$ is a ring possibly without identity when $E$ is an infinite graph but its maximal quotient ring $Q^r_{max}(L_K(E))$ is a ring with identity and therefore we can use the type theory of von Neumann regular right self-injective rings for ring $Q^r_{max}(L_K(E))$. 

The theory of types was first proposed by Murray and von Neumann \cite{MV} but it was developed as a classification tool later by Kaplansky in \cite{Kap1} for a certain class of rings of operators which are called Baer rings. Since von Neumann regular right self-injective rings are Baer rings, Kaplansky's theory is applicable to them. We recall basics of the structure theory of von Neumann regular right self-injective rings. A von Neumann regular right self-injective ring is said to be of type $I$ provided it contains a faithful abelian idempotent and it is said to be of type $II$ provided $R$ contains a faithful directly-finite idempotent but no nonzero abelian idempotents. A von Neumann regular right self-injective ring is of type $III$ if it contains no nonzero directly-finite idempotents. Moreover, a von Neumann regular right self-injective ring is said to be of type $I_{f}$ (resp., $I_{\infty}$) if $R$ is of type $I$ and is directly-finite (resp., purely-infinite) and it is said to be of type $II_{f}$ (resp., $II_{\infty}$) if $R$ is of type $II$ and is directly-finite (resp., purely-infinite). It is well known that any von Neumann regular right self-injective ring $R$ can be decomposed as a product $R=R_1 \times R_2 \times R_3 \times R_4 \times R_5$ where $R_1$ is of type $I_f$, $R_2$ is of type $I_\infty$, $R_3$ is of type $II_f$, $R_4$ is of type $II_\infty$, and $R_5$ is of type $III$ (see \cite{Goodearl}, pp. 111-115).

In \cite[Lemma 2(b)]{AS} it is shown a right non-singular right $\Sigma$-$V$ ring with identity has bounded index of nilpotence. However, there are some flaws in the argument there. So, we correct those mistakes in the first part of the proof below.    
 
\begin{prop} \label{bdd idx} 
If $L:=L_{K}(E)$ is a $\Sigma$-$V$ ring, then $L$ has bounded index of nilpotence.
\end{prop}

\begin{proof}
Since $L$ is non-singular, the maximal right quotient ring $Q$ of $L$ is a von Neumann regular right self-injective ring with identity. Now we claim that $Q$ is directly-finite. Assume to the contrary that $Q$ is not directly-finite. Then there exists an infinite family of orthogonal idempotents $\{e_{ij}: (i,j)\in \mathbb N\times \mathbb N\}$ in $Q$. For each $n\in\mathbb N$, we may produce by induction an idependent family of cyclic submodules $C_{n,i}$ of $L_L$, where $C_{n,i}\cong C_{n,j}$ for all $i,j=2,3, \ldots,n$ as done in \cite{AS}. For the sake of completeness, we give a sketch of the construction. Let $n>1$. Since $L\subseteq _{e}Q$, there exists a nonzero cyclic submodule $C_{n,1}$ of $L_L$ such that $C_{n,1}\subseteq e_{n^{2},n^{2}}Q\cap L$. Now we choose $0\neq x_{2}\in e_{n^{2}+1,n^{2}}C_{n,1}\cap L$. Then $x_{2}=e_{n^{2}+1,n^{2}}x_{1}, $ where $x_{1}\in C_{n,1}$. Denote $C_{n,2}=x_{2}L$ and redefine $C_{n,1}$ by setting $C_{n,1}=x_{1}L$. Define the module homomorphism $\varphi:C_{n,1}\longrightarrow C_{n,2}$ by $\varphi (x)=e_{n^{2}+1,n^{2}}x$. Clearly, $\varphi$ is an isomorphism (with inverse given by left multiplication by $e_{n^2, n^2+1}$), and so $C_{n,1}\cong C_{n,2}$. Suppose now that we have defined cyclic submodules $C_{n,1}\cong C_{n,2}\cong ...\cong C_{n,j-1}$ in $L_L$, where $C_{n,i}=x_{i}L$, $i=1,2,\ldots, j-1$. Next, we choose $x_{j}$ such that $x_{j}$ $\in e_{n^{2}+j-1,n^{2}+j-2}C_{n,j-1}\cap L$ and write $x_{j}=e_{n^{2}+j-1,n^{2}+j-2}x_{j-1}r_{j-1}$ where $r_{j-1}\in L$. Let $x_{j-1}^{^{\prime }}=x_{j-1}r_{j-1}$, and set $C_{n,j}=x_{j}L$. Now redefine $C_{n,j-1}=x_{j-1}^{^{\prime }}L$ (which is contained in the previously constructed $C_{n,j-1}$). Then $C_{n,j-1}\cong $ $C_{n,j}$ under the isomorphism that sends $x\in C_{n,j-1}$ to $e_{n^{2}+j-1,n^{2}+j-2}x$. We redefine preceding $C_{n,1},$ $C_{n,2},\ldots, C_{n,j-2}$ accordingly so that they all remain isomorphic to each other and to $C_{n,j-1}$. Note that the family $\{C_{n,i}$: $n=2,3,\ldots,$ \ $i=1,2,\ldots, n\}$ is independent since $\{e_{ij}Q: i, j\in \mathbb N\times \mathbb N\}$ is independent. By our construction, $C_{n,i}\cong C_{n,j}$ for all $n=2,3, \ldots$ and $1\leq i,j\leq n $. 

Now we choose maximal submodules $M_{n,i}$ of $C_{n,i},$ $n=2,3, \ldots$ and $1\leq i\leq n$, such that $C_{n,i}/M_{n,i}\cong C_{n,j}/M_{n,j}$ for all $n,i,j$. Set $M=\oplus _{n,i}M_{n,i}$. Then $(\bigoplus C_{n,i})/(\bigoplus M_{n,i}) \cong \bigoplus (C_{n,i}/M_{n,i})$ is a direct sum of isomorphic simple $L$-modules and so it is injective. Therefore, it is a direct summand of $L/M$, a contradiction. Therefore $Q$ must be directly finite. Now, we may invoke the type theory of von Neumann regular right self-injective rings and write $Q=Q_1\times Q_2$, where $Q_1$ is of type $I_f$ and $Q_2$ is of type $II_f$. As shown in \cite[Lemma 2(b)]{AS}, it turns out that $Q_2=0$ and so $Q$ is of type $I_f$ and consequently, $Q$ is a finite direct product of matrix rings over 
abelian regular self-injective rings. This shows $Q$ has bounded index of nilpotence and hence, $L$ has bounded index of nilpotence. 
\end{proof}

\begin{remark} \rm
Kaplansky asked if a prime von Neumann regular is primitive \cite{Kap}. This was answered in the negative by Domanov \cite{Domanov} who constructed a non-primitive prime von Neumann regular group algebra. Recently, in \cite{ABR} an example of a prime non-primitive von Neumann regular Leavitt path algebra has been constructed. Inspired by Kaplansky, Fisher raised the question whether a prime right $V$-ring is right primitive \cite{Fisher}. This question is still open. In view of the above proposition it can be seen that a prime $\Sigma$-$V$ Leavitt path algebra is always primitive. Because if a Leavitt path algebra $L$ is a prime $\Sigma$-$V$ ring, then its maximal right quotient ring $Q$ is a prime von Neumann regular ring with bounded index of nilpotence and hence $Q$ is simple artinian. This shows $L$ is right/left primitive. 
\end{remark}

The following lemma is implicit in \cite{R} and since it is useful in the
proof of Theorem \ref{main Th}, we state and give its simple proof.

\begin{lem} \label{ml}
Suppose $E$ is an acyclic graph. Then every
graded prime ideal $P$ of $L=L_{K}(E)$ with $P\cap E^{0}=H$ is of the form $P=I(H,B_{H})$ with $E^{0}\backslash
H$, a maximal tail.
\end{lem}

\begin{proof}
By \cite[Theorem 3.2]{R}, $P$ is of the form $P=I(H,B_{H})$
or $P=I(H,B_{H}\backslash\{u\})$ where $E^{0}\backslash H$ is a maximal tail
and, further, in the second case, $v\geq u$ for every $v\in E^{0}\backslash
H$. We claim that $P$ cannot be of the form $I(H,B_{H}\backslash\{u\})$.
Because, otherwise, since $E^{0}\backslash H$ is a maximal tail, there is an
edge $e$ with $s(e)=u$ and $r(e)=w\in E^{0}\backslash H$ and, since $w\geq u$,
this would then give rise to a closed path in $E^{0}\backslash H$,
contradicting the fact that $E$ contains no cycles. Thus $P=I(H,B_{H})$.
\end{proof}

We are now ready to prove the main theorem of this section.

\begin{theorem}
\label{main Th} Let $L$ be the Leavitt path algebra of an arbitrary graph $E$
over a field $K$. Then the following properties are equivalent:

(a) $L$ is a right/left $\Sigma$-$V$ ring;

(b) The graph $E$ contains no cycles, there is a positive integer $d$ such
that the number of distinct paths that end at any vertex $v$ in $E$ is less than or equal to $d$, the length of any path in $E$ is less than or equal to $d$ and every path in $E$ eventually
ends at a sink;

(c) The graph $E$ contains no cycles and every maximal tail $M$ of vertices in
$E$ is finite (and thus there is a fixed $w\in M$ such that $u\geq w$ for
every $u$ in $M$) and the number of distinct paths in $E$ that end at $w$
(including $w$) is at most a fixed positive integer $d$ which is independent
of $M$;

(d) $L$ is von Neumann regular and there is a fixed positive integer $d$ such
that, such that for any prime ideal $P$ of $L$, $\ L/P\cong M_{n}(K)$ with
$n\leq d$.
\end{theorem}

\begin{proof}
Assume (a). Since contains no cycles. Also, by Proposition \ref{bdd idx},
$L$ has bounded index say $d$. We then appeal to Theorem \ref{bdd index} to
conclude that there is a fixed positive integer $d$ such that the number of
distinct vertices in any path in $E$ is less than or equal to $d$ and that the number of
distinct paths that end at any given vertex $v$ (including $v)$ is less than or equal to $d$.
Now, by Proposition \ref{df}, $L$ is directly-finite and so, by
Lemma \ref{df=>NE}, no closed path in the graph $E$ has an exit. On the other
hand, since $L$ is also a right $V$-ring, Lemma \ref{wr} implies that $L$ is
right weakly regular. Now, in view of the Lemma \ref{k}, we know that the
graph $E$ satisfies Condition (K) which in particular implies that every cycle
in $E$ has an exit. These contradicting statements imply that the graph $E$
contains no cycles. This proves (b).

Assume (b). Let $M\subseteq E^{0}$ be a maximal tail. Since there are no
cycles in $M$ and since every path has length at most $d$, every path
eventually ends at a sink. In addition, $M$ is downward directed and so $M$
has a unique sink, say $w$ and $u\geq w$ for every vertex $u$ in $M$. Since
the number of distinct paths ending at any vertex (and, in particular, at $w$)
is $\leq d$ and since no path in $E$ will have length greater than $d$, $M$ must be finite
and every vertex $v\in M$ emits at most finitely many edges $e$ such that
$r(e)\in M$. This proves (c).

Assume (c). Since $E$ is acyclic, $L$ is von Neumann regular by (Theorem 1,
\cite{AR}). Let $P$ be any prime ideal of $L$. Now $P$ is a graded ideal as
$E$ is acyclic. Hence $P=I(H,B_{H})$, by Lemma \ref{ml} where $H=P\cap E^{0}$
and $M=E^{0}\backslash H$ is a maximal tail. By hypothesis, $M$ is finite,
contains a vertex $w$ such that $v\geq w$ for every $v$ in $M$. Now
$E\backslash(H,B_{H})$ is acyclic and $((E\backslash H,B_{H}))^{0}%
=E^{0}\backslash H=M$ is downward directed. So $w$ is a unique sink in
$E\backslash(H,B_{H})$ and, as the number of distinct paths ending at $w$ is
$\leq d$, $E\backslash(H,B_{H})$ is finite (acyclic) graph and we appeal to
(Proposition 3.5, \cite{AAS}), to conclude that $L/P\cong$ $L_{K}%
(E\backslash(H,B_{H}))\cong M_{r}(K)$, where $r\leq d$. This proves (d).

Assume (d). Suppose $S$ be a simple right $L$-module. Now $P=ann_{L}(S)$ is a
prime ideal and so, by hypothesis, $L/P\cong M_{n}(K)$ for some integer $n$.
Now $L/P$ is semisimple artinian and so $S$ is a right $\Sigma$-injective
$L/P$-module. As $L$ is von Neumann regular, $S$ is also a right $\Sigma$-injective $L$-module (see \cite{Lam}). This proves (a).
\end{proof}

When the graph $E$ is row-finite, Theorem \ref{main Th} can be strengthened
leading to a structure theorem for Leavitt path algebras of row-finite graphs
which are $\Sigma$-$V$ rings.

As a consequence, for a Leavitt path algebra over a row-finite graph we have the following. 

\begin{theorem}
\label{row-finite} Let $E$ be a row-finite graph. Then the following
properties are equivalent for $L:=L_{K}(E)$:

(i) $L$ is a $\Sigma$-$V$ ring;

(ii) There is a fixed positive integer $d$ and an isomorphism%

\[
L\cong%
{\displaystyle\bigoplus\limits_{i\in I}}
M_{n_{i}}(K)
\]
where $I$ is an arbitrary index set and $n_{i}\leq d$ for all $i\in I$. Thus,
in particular, $L$ is semi-simple (that is, a direct sum of simple left/right ideals).
\end{theorem}

\begin{proof}
Assume (i). By Proposition \ref{bdd idx}, $L$ has bounded index of nilpotence, say $d$.
Since the graph $E$ contains no cycles, it then follows from Theorem \ref{bddIdx-row-finte} that
$L\cong%
{\displaystyle\bigoplus\limits_{i\in I}}
M_{n_{i}}(K)$ where $I$ is an arbitrary index set and $n_{i}\leq d$ for all
$i\in I$. This proves (ii).

Now (ii)$\implies$(i) follows immediately, since $L$, being a ring direct sum of the (semisimple artinian) matrix rings $M_{n_{i}}(K)$, is a $\Sigma$-$V$ ring.
\end{proof}

We now construct an example illustrating the ideas of Theorem \ref{main Th}.
It also shows that Theorem \ref{row-finite} no longer holds if the graph $E$ is not row-finite.

\begin{example} \rm
\label{Ex-1}Consider the following ``infinite clock" graph $E$:%
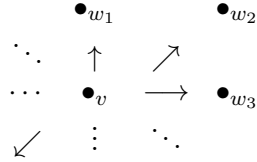
\[%
\begin{array}
[c]{ccccc}
&  & \bullet_{w_{1}} &  & \bullet_{w_{2}}\\
& \ddots & \uparrow & \nearrow & \\
& \cdots & \bullet_{v} & \longrightarrow & \bullet_{w_{3}}\\
& \swarrow & \vdots & \ddots & \\
_{{}} &  &  &  &
\end{array}
\]
Thus $E^{0}=\{v\}\cup\{w_{1},w_{2},\cdot\cdot\cdot,w_{n},\cdot\cdot\cdot\}$
where the $w_{i}$ are all sinks. For each $n\geq1$, let $e_{n}$ denote the
single edge connecting $v$ to $w_{n}$. The graph $E$ is acyclic. The number of
distinct paths ending at any given sink \ (including the sink) is $2$ and all
the paths eventually end at a sink. Thus, by Theorem \ref{main Th}, $L_{K}(E)$
is a $\Sigma$-$V$ ring. For each $n\geq1$, $H_{n}=\{w_{i}:i\neq n\}$ is a
hereditary saturated set, $B_{H_{n}}=\{v\}$ and $E^{0}\backslash
H_{n}=\{v,w_{n}\}$ is downward directed. Hence the ideal $P_{n}$ generated by
$H_{n}\cup\{v-e_{n}e_{n}^{\ast}\}$ is a prime ideal and $L_{K}(E)/P\cong
M_{2}(K)$. Clearly $%
{\displaystyle\bigcap\limits_{n=1}^{\infty}}
P_{n}=0$. Moreover, every prime ideal $P$ of $L_{K}(E)$ is equal to $P_{n}$
for some $n$. Clearly, $L_{K}(E)$ is a subdirect product of countably-infinite
number of copies of $M_{2}(K)$. Since $E$ contains no cycles, $L_{K}(E)$ is
also von Neumann regular by \cite{AR}.

But $L_{K}(E)$ cannot decompose as a direct sum of the matrix rings $M_{2}%
(K)$. Because, otherwise, $v$ would lie in a direct sum of finitely many
copies of $M_{2}(K)$. Since the ideal generated by $v$ is $L_{K}(E)$,
$L_{K}(E)$ will then be a direct sum of finitely many copies of $M_{2}(K)$.
This is impossible since $L_{K}(E)$ contains an infinite set of orthogonal
idempotents $\{e_{n}e_{n}^{\ast}:n\geq1\}$.

We can also describe the internal structure of this ring $L_{K}(E)$. The socle
$S$ of $L_{K}(E)$ is the ideal generated by the sinks $\{w_{i}:i\geq1\}$,
$S\cong%
{\displaystyle\bigoplus\limits_{\aleph_{0}}}
M_{2}(K)$ and $L_{K}(E)/S\cong K$.
\end{example}

\begin{remark} \rm
In describing when a Leavitt path algebra $L_{K}(E)$ is a $\Sigma$-$V$-ring, the
proof of Theorem \ref{main Th} uses the fact that if $L_{K}(E)$ is directly-finite,
then no cycle in the graph $E$ has an exit. Conversely, if no cycle in a
graph $E$ has an exit, then $L_{K}(E)$ indeed becomes directly-finite. This
fact was proved in \cite{V} with a lengthy proof using interesting concepts
of traces and the Cohen-Leavitt path algebras. Because of the relevance of
this result to our investigation, we wish to provide an alternative shorter
proof of this result below.
\end{remark}

 We provide the altrnative proof by using the ideas developed in this section and also
using the subalgebra construction given in \cite{AR}. This subalgebra
construction has turned out to be a very useful tool in making the
``local-to-global jump" while proving a ring theoretic property of finite
character for a Leavitt path algebra $L$ of an arbitrary graph $E$, as
exemplified in proving the following theorems in a series of papers:

1)  $L$ von Neumann regular $\Leftrightarrow$ $E$ is acyclic(\cite{AR}), (2) Every simple left/right $L$-module is graded $\Leftrightarrow$ 
$L$ is von Neumann regular (\cite{HR}), (3) Every Leavitt path algebra $L$ is
a graded von Neumann regular ring (\cite{H}), (4) Every Leavitt path algebra
$L$ a right/left Bezout ring \cite{AMT}. 

Proposition 2 in \cite{AR}, as stated, does not include the additional
properties implied by the subalgebra construction. Indeed, a careful
inspection of the construction in \cite{AR} shows that the morphism $\theta$
in the construction is actually a graded morphism whose image is a graded
submodule of $L$ and it also reveals some properties of cycles. We include
these facts in the following stronger formulation of Proposition 2  of
\cite{AR} which we shall be using.

\begin{theorem}
\label{GeneRanga} (\cite{AR}) Let $E$ be an arbitrary graph. Then the Leavitt
path algebra $L:=L_{K}(E)$ is a directed union of graded subalgebras
$B=A\oplus K\varepsilon_{1}\oplus\cdot\cdot\cdot\oplus K\varepsilon_{n}$ where
$A$ is the image of a graded homomorphism $\theta$ from a Leavitt path algebra
$L_{K}(F_{B})$ to $L$ with $F_{B}$ a finite graph (depending on $B$), the
elements $\varepsilon_{i}$ are homogeneous mutually orthogonal idempotents and
$\oplus$ denotes a ring direct sum. Moreover, any cycle $c$ in the graph
$F_{B}$ gives rise to a cycle $c^{\prime}$ in $E$ such that if $c$ has an exit
in $F_{B}$ then $c^{\prime}$ has an exit in $E$.
\end{theorem}

We now provide one more instance of using Theorem \ref{GeneRanga}  to provide
an alternative shorter proof of the following theorem by V\'{a}s \cite{V}.

\begin{theorem}
\label{ne = > df}(\cite{V}) Let $E$ be an arbitrary graph in which no cycle
has an exit. Then $L:=L_{K}(E)$ is directly finite.
\end{theorem}

\begin{proof}
Case 1: Suppose $E$ is a finite graph. Let $P$ be a graded prime ideal of $L$
with $P\cap E^{0}=H$. By \cite{R}, $E^{0}\backslash H$ is downward directed.
If $E^{0}\backslash H$ contains a cycle, then our hypothesis implies it is the
only cycle (without exits) based at a vertex $v$ and moreover, every path in
$E\backslash H$ eventually ends at \ $v$. Then, by Proposition 3.4 of
\cite{AAS}, $L/P\cong L_{K}(E\backslash H)\cong M_{t}(K[x,x^{-1}])$ for some
positive integer $t$. On the other hand, if $E\backslash H$ contains no
cycles, it will have a unique sink $w$ such that $u\geq w$ for all $u\in
E^{0}\backslash H$. Then by Proposition 3.5 of \cite{AAS-1}, $L/P\cong
L_{K}(E\backslash H)\cong M_{t}(K)$ for some integer $t$. Since the
intersection of all graded prime ideals is zero, we have an embedding of $L$
into a direct product\ $D$ of finitely many matrix rings \ $A_{i}$,
$i=1,\cdot\cdot\cdot,n$, where each $A_{i}\cong$ $M_{t}(K)$ or $M_{t}%
(K[x,x^{-1}]$. Since the matrix rings $M_{t}(K)$ and $M_{t}(K[x,x^{-1}]$ are
directly-finite, the ring $D$ and hence $L$ is directly-finite.

Suppose $E$ is an arbitrary graph. Let $a,b\in L$ with a local unit $u$ so
that $ua=au=a$, $ub=bu=b$ and $ab=u$. We wish to show that $ba=u$. By Theorem
\ref{GeneRanga}, there is a graded subalgebra $B$ such that $a,b,u\in B$ and
$B=A\oplus K\varepsilon_{1}\oplus\cdot\cdot\cdot\oplus K\varepsilon_{n}$ where
$A$ is the image of a graded subalgebra $B$ such that $a,b,u\in B$ and
$B=A\oplus K\varepsilon_{1}\oplus\cdot\cdot\cdot\oplus K\varepsilon_{n}$ where
$A$ is the image of a graded homomorphism $\theta$ from a Leavitt path algebra
$L_{K}(F_{B})$ to $L$ with $F_{B}$ a finite graph (depending on $B$), the
elements $\varepsilon_{i}$ are homogeneous mutually orthogonal idempotents and
$\oplus$ denotes a ring direct sum. Clearly, $B\cong L_{K}(G)$ where
$G=F_{B}\cup\{$isolated vertices $v_{1},\cdot\cdot\cdot,v_{n}$ corresponding
to $\varepsilon_{1},\cdot\cdot\cdot,\varepsilon_{n}\}$. Since no cycle in $E$
has an exit, no cycle in $F_{B}$ has an exit by Theorem \ref{GeneRanga} and
so no cycle in the finite graph $G$ has an exit. So we appeal to Case 1 to
conclude that $B$ is directly-finite. As noted in the Preliminaries section,
we then have $ba=u$.
\end{proof}

\begin{remark} \rm
We wish to point out an interesting consequence of our proof: A
directly finite Leavitt path algebra $L$ of an arbitrary graph $E$ is a
directed union of graded subalgebras each of which is a directly-finite
Leavitt path algebra of a finite graph.
\end{remark}

\end{document}